\documentclass[a4paper]{article}
\usepackage{caption}
\usepackage{subcaption}
\usepackage{calc}
\usepackage[a4paper,width=19cm,height=22.5cm]{geometry}
\usepackage[ngerman,english,spanish]{babel}
\usepackage{etex}
\usepackage[T1]{fontenc}
\usepackage[utf8]{inputenc}
\usepackage{ellipsis,ragged2e}
\usepackage{microtype}
\usepackage{amsmath}
\usepackage{amssymb}
\usepackage{amsthm}
\usepackage{color}
\definecolor{rasp}{rgb}{.89,.04,.36}      
\usepackage{times}
\usepackage{pictexwd}
\usepackage[pdftex]{graphicx}
\usepackage{graphicx}
\usepackage{caption}
\usepackage{subcaption}
\usepackage{pst-pdf,pstricks-add}
\usepackage{tikz}
\usetikzlibrary{plotmarks}
\usepackage{pgfplots}
\usepackage{dsfont}
\usepackage{kpfonts}
\usepackage{listings}
\usepackage{cite}
\usepackage{mathrsfs,graphicx,color,float, indentfirst,textcomp}
\usepackage{setspace}
\usepackage{latexsym,amssymb,lscape,amsmath,amsthm,amsfonts,amssymb,cite,enumerate}
\usepackage{lscape}
\usepackage{algorithm} 
\usepackage{algorithmic} 
\usepackage{multirow} 
\usepackage{pifont}
\usepackage{doi}
\usepackage{xfrac}
\usepackage{indentfirst}
\usepackage{algorithm} 
\usepackage{algorithmic}
\usepackage{multirow} 
\usepackage{xcolor}
\usepackage{geometry}
\usepackage{booktabs}
\usepackage{tikz}
\usepackage{pgfplots}
\pgfplotsset{compat=newest}

\newtheorem{theorem}{Theorem}[section]

\newtheorem{lemma}[theorem]{Lemma}

\title{ Model-Order Reduction For Hyperbolic Relaxation Systems
}
\date{\today}
\author{Sara Grundel \footnote{grundel@mpi-magdeburg.mpg.de}\\ 
	{\small\it Max-Planck-Institut f\"{u}r Dynamik komplexer technischer Systeme} \\
	{\small\it Sandtorstr. 1, 39106 Magdeburg, Germany
	}\and  
	Michael Herty \footnote{herty@igpm.rwth-aachen.de} \\
	{\small\it Institut f\"{u}r Geometrie und Praktische Mathematik (IGPM)} \\
	{\small\it RWTH Aachen University} \\
	{\small\it Templergraben 55, 52062 Aachen, Germany
}
}

\newcommand{\R}{\mathbb{R}}

\begin{document}
	\selectlanguage{english}
	\date{\today}
	\maketitle
	\begin{abstract}
	We propose a novel framework for model-order reduction of hyperbolic differential equations. The approach combines a relaxation formulation of the hyperbolic equations with a discretization using shifted base functions.  Model-order reduction techniques are then applied to the resulting system of coupled ordinary differential equations. On computational examples including in particular the case of shock waves we show the validity of the approach and the performance of the reduced system.   
	\end{abstract}
\section{Introduction}

Model-order reduction has been successfully applied to large-scale systems of ordinary differential equations as well as problems
governed by elliptic or parabolic differential equations, see e.g. \cite{Quarteroni2011,Haasdonk2008,haasdonk2008reduced,morAntSG01,morAntBG20,morMadPT02,morGre05}. 
There is a large variety of methods out there for these problem classes and all of them base on the idea that the solution space as a subset of either a large finite dimensional space or possibly an infinite dimensional function space is well approximated by a finite dimensional linear subspace of relatively low dimension. There are several different methods to determine a suitable subspace and several  methods to use it for a reduced order model. Some model order reduction methods only take the description of the system to create the projection onto that subspace, and some use  data created from solving the full system. A crucial point in the interest and usefulness of a reduced model is that one is not interested in one single solution for one single equation but for a collection of solutions or equations. Sometimes this collection is created by a parameter in the differential equation, sometimes by a varying input function or  by considering different starting values.

A way  to quantify how reducible an equation is can be done by understanding how well the solution space is approximated by the best  $n$-dimensional linear subspace. This concept is referred to as the Kolmogorov $n$-width in  the literature. This is also studied for specific hyperbolic problems and the best approximation space in this setting is not satisfying. Therefore unfortunately we need to rethink the general strategy for \textbf{nonlinear} hyperbolic problem. So far a general method is not available. Several approaches have been proposed to provide a suitable finite dimensional approximation space. In particular, in the case of linear hyperbolic system the solution can be expressed as linear semigroup on suitable spaces and therefore an approximation by finite dimensional subspaces is feasible \cite{himpe2014,morSarGB20,morMoSTetal01,morXiaM13,laakmann2021efficient}. For linear hyperbolic problems the transport speed is constant and known a priori. This allows to exploit the idea of shifted base functions. Several different approaches exist and they have partially been extended to the nonlinear case  \cite{morMoSTetal01,morXiaM13,reiss2018shifted,haasdonk2008reduced}.  In the nonlinear case a major obstacle has been the loss of regularity of the solutions in the presence of shocks. Those  also move at a speed determined through a possibly nonlinear relation out of the solution itself. This time-dependence in the approximate finite dimensional space has been dealt with by time dependent space transformation as part of the reduced model. There is a large body of literature addressing different solutions to this well established problem  \cite{PeterBook,RBHyp,Benjamin2018model, Welper2017,Metric2019,Nair,Cagniart2019}.  They all use very different ways to deal with the creation of a non-linear subspace approximating the solution space. 
\medskip 

We propose  a method to treat loss of regularity due to shocks as well as the nonlinear transport speed. To that end we  first lift the solution space and then find  a linear subspace exploiting known techniques. The lifting is done in  two steps, first a hyperbolic relaxation \cite{Jin1995b,Natalini1999,Bianchini2006} and then a discretization using suitable spacetime Ansatz functions. The hyperbolic relaxation methods use a suitable reformulation of the nonlinear flux at the expense of an enlarged system. This in turn allows to keep possible discontinuous solutions but reduces the transport part to a linear transport. The linear part ensures further that the new system formally has  fixed transport speeds. The latter system is therefore amendable for treatment within model order reduction as shown in this work.  We propose to capture the movement of discontinuities by a suitable {\em moving} approximations. On those approximations we perform a suitable model order reduction. Based on the continuous formulation we discuss possible numerical discretizations and show computational results in the case of shocks.   

\setcounter{tocdepth}{1}
\tableofcontents

\section{Reducability of Scalar non-linear hyperbolic equations}

We consider a scalar nonlinear hyperbolic differential equation for the unknown $U=U(t,x)$ on the torus $\mathbb{T}=[-1,1] \subset \R$  as solution to  
\begin{align}\label{equation} 
\partial_t U(t,x) +\partial_x f(U(t,x)) = 0
\end{align}
subject to the initial conditions $u_0: \mathbb{T} \to \R$ 
\begin{align}
\label{ic}
u(0,x)=u_0(x).
\end{align}
The flux function $f \in C^2(\R;\R)$ is assumed to be  nonlinear.  
Even for smooth initial data \eqref{ic} $u_0$ the solution $u$ may exhibit discontinuities in finite time \cite{Dafermos:2005aa}.  Therefore,
weak entropy solutions to \eqref{equation} have been introduced and we refer
to \cite{Dafermos:2005aa} for more details on well-posedness of weak solutions. 
This presentation is concerned with finding a reduced model to this system in the sense of approximating  the solution on a lower dimensional manifold. For (linear) elliptic differential equations the lower dimensional manifold can be shown to be a linear subspace and the model-order reduction can be successfully applied. However, for nonlinear hyperbolic systems this approach is not straight forward as  already mentioned in the introduction. For general nonlinear problems the typical way to create a reduced order model is  to first solve the system  at certain instances (in time) using a   high dimensional solution technique. This information is used  to define a linear subspace of the solution space which becomes the search space in which the equation is then solved resulting in a so called reduced system, which is then used to approximate the solution for different parameters or different input functions.  For our setting we assume the flux function $f$ is given, however the initial condition given by $u_0$ could vary. Therefore, a suitable reduced modeling technique should allow to generate  a reduced system  which is able to approximate the solution to the original equation for different initial conditions.
\par 
In order to derive the discretization with the space-time ansatz function on the relaxation  we consider as an example the linear case first where we already have a linear transport operator. Let 
\begin{align}
 f(U)=\lambda U \end{align}  
 with coefficient $\lambda \not = 0 $. The explicit solution to \eqref{equation}, 
 \eqref{ic} on $\mathbb{T}$ is given by 
 \begin{align}\label{explicit}
 U(t,x)=u_0(x-\lambda t ). 
 \end{align}

 Classical model order reduction of partial differential equations is based on the idea that the numerical solution is computed as an approximation of the form
 \begin{align}
 U(t,x)\approx \sum u_j(t)\phi_j(x),
 \end{align}
 for a set of basis functions $\phi_j$, like for example a  finite element space. In general reduced solutions to the PDE are also described in a similar fashion but with different  basis functions, which are picked in such  a way that we do not need so many by exploiting the structure of the given equation. In other words model reduction  tries to extract a lower dimensional space within the FE space which represents the solution of the given problem well. Assuming that we choose $u_0$ to be a compactly supported local finite element basis function, the solution $U(t,x)$ which is just the transported $u_0$ has a support that moves through the entire space over time. The collection of this functions evaluated at discrete time instances would fast span a large dimensional space within the finite element space leaving not  too much hope that we can find a low dimensional subspace. This has been recognized as a problem for hyperbolic systems for a while \cite{RBHyp} and a few techniques have been used to overcome that. The most promising approach being to use an Ansatz where the basis function contain a time dependent spatial shift. For linear problem as the speed is clear this can be done easily, and for nonlinear the spatial transformation is part of the hard work of finding  the right reduced system and is still a work in progress, but with progress for certain problems \cite{Pal87}. In this paper we use this idea of the spatial shift  not to create the right reduced order model but to discretize the full model in order to get a large scale ordinary differential equation that no longer suffers from the transport phanomena.  Our large dimensional ansatz space is given by a set of basis function $\phi_j$ but evaluated at a fixed spatial shift:
 
  \begin{align}\label{final linear}
 U(t,x) = \sum\limits_{j=1}^N u_{j}(t) \phi_j(x-\lambda t). 
 \end{align}
 And in fact if the initial condition   $u_0$ is expanded in a truncated series of $N$ coefficients 
 \begin{align}\label{u0:lin}
 u_0(x) = \sum\limits_{j=1}^N u_{0,j} \phi_j(x) ,
 \end{align} 
 for some functions $\{  \phi_j \}_{j=1}^N$. The explicit solution \eqref{final linear} then yields  the {\em exact }
  solution for $u_j(t)=u_{0,j}$ on the linear transport problem.

 If we choose a  function $u_0(x)$ as the correct linear combination of our linear subspace as in equation \eqref{u0:lin} and we consider the solution $u(t,x)$ within the one-dimensional manifold spanned by $u_0$
 \begin{align}
 \{ u(x,t) = u_0(x-\lambda t \}
 \end{align}
 we obtain the exact solution. While this approach can be extended to linear transport equations with nonlinear right-hand side, as e.g. 
 \begin{align}\label{eq:linrhs}
 \partial_t u(t,x)+ \lambda \partial_x u(t,x)=g(u(t,x)),
\end{align} 
 this approach however does {\bf not} extend to a non-linear equations.  It is important to note that the previous approach only  works because $\lambda$ is constant. However, in 
 the case of nonlinear flux $U \to f(U)$ the characteristic  $\frac{dx}{dt}(t)$ depends
  on the value of the initial datum $u_0$  at $x_0:$
\begin{align}
\frac{dx}{dt}=f'(U(t,x(t)),\;  x(0)=x_0 \mbox{ and } U(t,x(t))=u_0(x_0).
\end{align}
In the nonlinear case it is challenging to determine the correct shift. There is a big effort to do so  in the literature and for certain problems this has been applied successfully~\cite{Benjamin2018model,Welper2017,Metric2019,Nair,Cagniart2019}. In the following we want to develop a general method allowing to have a fixed shift in the base functions.

\section{Semi--Discretization Compatible to Model Order Reduction
}

Our approach is somewhat more robust with respect to the type of nonlinear function and the initial condition used as we do not have to track the speed or possibly more than one speed if the waves travel in different directions. The first ingredient is a stiff relaxation approximation \eqref{relax} considered e.g. in \cite{Liu2001,Chalabi:1999aa,Aregba-Driollet1996,Jin1995b,Natalini1999,Bianchini2006,Yong:2008aa,Pareschi2003,Pareschi2005,Klar1999}. 
  For the scalar problem\eqref{equation} the relaxation approximation reads
 \begin{align}\label{relax}
\partial_t u(t,x)+\partial_x v(t,x) &=0\\
\partial_t v(t,x)+\lambda^2\partial_x u(t,x) &=-\frac{1}{\epsilon}(v(t,x)-f(u(t,x))). 
\end{align}
Here, $\lambda>0$ is a positive fixed parameter that fulfills the subcharacteristic condition 
\begin{align}
\label{subcharacteristic}
\lambda \geq \max\limits_{x \in \mathbb{T}} |f'(u_0(x)) |
\end{align}
and $\epsilon>0$  is the (small) relaxation parameter. At the expense of an additional variable $v=v(t,x)$ 
the relaxation system \eqref{relax} introduces a {\em linear, hyperbolic} approximation to equation \eqref{equation}. 
Using a Chapman--Enskog expansion in $\epsilon$ a formal computation shows that 
\begin{align}\label{viscous}
\partial_t u(t,x) + \partial_x f(u(t,x)) = \epsilon \partial_x \left(   (\lambda^2 - f'(u(t,x))^2) \partial_x u(t,x) \right) + O(\epsilon^2).
\end{align}
Hence, $u$ given by \eqref{relax} is a viscous approximation to the  solution $U$ of equation \eqref{equation}. However, 
it needed to be pointed out that \eqref{relax} is linear hyperbolic and therefore a similar decomposition as shown above
might be possible. 
\par 
The eigenvalues of the linear part in equation \eqref{relax}  are $ \lambda$ and $-\lambda$, respectively. For small values of $\epsilon$  we expect $v \approx f(u)$ and therefore we set the following initial conditions for $(u_0,v_0)$ 
\begin{align}\label{ic relax}
u(0,x)=u_0(x) \mbox{ and } v(0,x)=f(u_0(x)).
\end{align}

Diagonalizing the system \eqref{relax} using the 
variables 
\begin{align}
w^\pm(t,x)=v(t,x) \pm \lambda \; u(t,x)
\end{align}
 and 
\begin{align}\label{trafo}
v(t,x)=\frac{1}{2}(w^+(t,x) +w^-(t,x)), \;
u(t,x)=\frac{1}{2\lambda} (w^+(t,x) - w^-(t,x) ), 
\end{align}
 respectively, yields the following system
 \begin{align} \label{w eq}
 \partial_tw^+  + \lambda \partial_x w^+ =-\frac{1}{\epsilon}\left ( \frac{w^++w^-}{2}-f\left (\frac{w^+-w^-}{2\lambda}\right)\right),\\
 \partial_tw^-  - \lambda \partial_x w^- =-\frac{1}{\epsilon}\left ( \frac{w^++w^-}{2}-f\left (\frac{w^+-w^-}{2\lambda}\right)\right),
 \end{align}
Their  corresponding initial conditions are
\begin{align}\label{ic w}
w^+(0,x)=f(u_0(x))+\lambda u_0(x)\quad w^-(0,x)=f(u_0(x))-\lambda u_0(x). 
\end{align}
Following the procedure of the linear case we introduce 
 $\{ \phi_j(\cdot) \}_{j=1}^N$ a set of $N$ differentiable functions $\phi_j: \mathbb{T} \to \R$ for $j=1,\dots, N.$  
The initial data $w_0^\pm$ is then expanded using the truncated series 

\begin{align}\label{ic w_e}
	w^\pm_0(x)= \sum\limits_{j=1}^N  \alpha^\pm_j(t) \phi_j(x), 
\end{align}
and the solution is expanded using the translated base functions 
 \begin{align}\label{Ansatz}
w^+(t,x) \approx \sum\limits_{j=1}^N   \alpha^+_j(t) \phi_j(x-\lambda t) \mbox{ and } 
w^-(t,x) \approx \sum\limits_{j=1}^N   \alpha^-_j(t) \phi_j(x+\lambda t), 
\end{align}
respectively. Note that in the case $f(u)=a u$ we in fact have that \eqref{Ansatz} is exact.  However, 
due to the nonlinearity of the right-hand side of \eqref{w eq} and contrary to the linear case 
the previous ansatz \eqref{Ansatz} is in general  {\em not} 
the exact solution to \eqref{w eq} and \eqref{ic w}. 
\par 
A series expansion of the original variables $(u,v)$ is obtained applying the linaer transformation \eqref{trafo}. Hence, using ansatz \eqref{Ansatz} in equation \eqref{relax}
we obtain 
\begin{align}
	\partial_t u + \partial_x v &=\frac{1}{2\lambda} \left( 
	\sum_{j=1}^N \dot{\alpha}^+_j(t)\phi_j(x-\lambda t)-\sum_{j=1}^N \dot{\alpha}^-_j(t)\phi_j(x+\lambda t) \right)=0,\label{linear1}
	\\
	\partial_t v +\lambda^2  \partial_x u &= \frac{1}{2}
	\left(\sum_{j=1}^N \dot{\alpha}^+_j(t)\phi_j(x-\lambda t)+\sum_{j=1}^N \dot{\alpha}^-_j(t)\phi_j(x+\lambda t) \right)=- \frac{1}{\epsilon}(v-f(u)),  \label{temp1}
\end{align}
where we did not expand $v$ and $u$ in terms of $\phi_j$  in the right-hand side of equation \eqref{temp1} for the sake of readability.
Define  the family of  matrices $t \to M(t) \in\R^{N,N}$ 
by 
\begin{align}
\label{mass matrix}
M_{jk}(t) = \int\limits_{ \mathbb{T}} \phi_j(x) \phi_k(x+2\lambda t) dx, \; t \geq 0,
\end{align}
and the projected initial data $b_k^\pm$ for $k=1,\dots,N$ as 
\begin{align}
	b_k^\pm = \int_{\mathbb{T}} \phi_k(x) \left( f\left(u_0(x)\right) \pm \lambda u_0(x) \right) dx. 
\end{align}
Then, the following system for the evolution of the coefficients $\alpha^\pm = (\alpha^\pm_j)_{j=1}^N$ is obtained
\begin{align}
M(0)\dot{\alpha}^+(t)  - M(t)\dot{\alpha}^-(t) &=0\label{eq:ode1}\\
M(0)\dot{\alpha}^+(t) +M(t) \dot{\alpha}^-(t) &= - \frac{2}{\epsilon} \left( 
\frac12 \left(M(0)\alpha^+(t) +M(t) \alpha^-(t) \right) - \tilde{F}(t, \alpha^\pm(t) )
\right) \label{eq:ode2} , 
\end{align}
where $\tilde{F}=(\tilde{F}_1,\dots,\tilde{F}_N)$ and 
\begin{align}
\tilde{F}_j(t, \alpha^\pm(t)) &:= \int_{\mathbb{T}}   \phi_j(x) f(   
\tilde{u}(t, x+\lambda t)   )dx   \label{tildeF},  \\
\tilde{u}(t,x) &:= \frac{1}{2\lambda} \left(\sum\limits_{j=0}^N \alpha_j^+(t) \phi_j(x-\lambda t ) -\alpha_j^-(t) \phi_j(x+\lambda t) \right).  
\label{tildeu}
\end{align} 
This is a result of multiplying \eqref{temp1} and \eqref{linear1}  by $\phi_j(x-\lambda t)$ for all $j$ and integrating it over $x$ on $\mathbb{T}$.
The initial data is given by
\begin{align}\label{ic red}
M(0)\alpha^+(0)=b^+ \mbox{ and } M(0)\alpha^-(0)=b^-,
\end{align}
which also follows from multiplying by said basis function and integration. 
Summarizing, for fixed $N$ and $\epsilon>0$, the stiff system \eqref{eq:ode1}-\eqref{eq:ode2} and \eqref{ic red} determine  the coefficients $\alpha^\pm(t)$ and $u$ given by equation \eqref{Ansatz} and \eqref{trafo}, i.e., 
\begin{align}\label{u relax N}
u^N(t,x) = \frac{1}{2\lambda} \left(\sum\limits_{j=0}^N \alpha_j^+(t) \phi_j(x-\lambda t ) -\alpha_j^-(t) \phi_j(x+\lambda t) \right).  
\end{align}
Note that it is not clear a priori if $M(t)$ for $t\geq 0$ is invertible and therefore the governing equations are not necessarily an ordinary differential equation, but possibly a differential algebraic equation. This point will be discussed in more detail in the forthcoming section. For the further considerations assume

\begin{align}\label{asss0}
(Assumption) \; \forall t \geq 0: \; M(t) \mbox{ is invertible. }
\end{align}

Summarizing, under assumption \eqref{asss0} the system \eqref{eq:ode1}-\eqref{eq:ode2} with initial conditions \eqref{ic red} yield the approximation \eqref{u relax N} to the solution $U=U(t,x)$ of the nonlinear conservation law \eqref{equation} on $\mathbb{T}$.  The proposed approximation \eqref{u relax N}  contains different approximation errors that have to be addressed in a numerical scheme. 
First, the solution is projected on the space spanned by the $N$ functions $\phi_j.$ Since we expect discontinuities the choice of suitable
functions $\phi_j$ is critical to the approximation error. Second, the derivation shows that $u^N$ given by \eqref{u relax N} in fact approximates the relaxation solution $u$ to the system \eqref{relax} for some fixed $\epsilon$. However, analytically, the sequence of weak solution $u^\epsilon$ to  equation \eqref{relax} converges weakly to the weak solution $U$  to equation \eqref{equation} as $\epsilon \to 0$ \cite{Bianchini2006}. The interplay of the obtained numerical errors with the choice of the parameters $\epsilon$ and $N$ will be investigated in the numerical results below.

\subsection{Properties of the system \eqref{eq:ode1}-\eqref{ic red} }

\renewcommand{\a}{\alpha}

Using the notation $\a = (\alpha^+, \alpha^-)$‚ we obtain 
\begin{equation}\label{reform}
	\begin{pmatrix}
		M(0) & - M(t) \\ M(0) & M(t) 
	\end{pmatrix} \frac{d}{dt}\a(t) =-\frac{1}\epsilon \begin{pmatrix} 0 \\  [M(0), M(t)] \a(t) -2 \tilde{F}(t,\a(t))  \end{pmatrix}. 
\end{equation}
The left hand side of equation \eqref{reform} consists of a $2\times 2$ block matrix. This matrix is invertible provided that  for all $t \geq 0$ 
$M(t)$ is invertible. In this case the inverse is explicitly given by 
\begin{equation}\label{inv R}	\frac12 \begin{pmatrix}
	M^{-1}(0) &  M^{-1}(0) \\ M^{-1}(t) & M^{-1}(t) 
\end{pmatrix} \end{equation}
 By suitable choice of $\{ \phi_j(\cdot) \}_j$ we can guarantee that $M(0)$ is invertible. In fact, if for all $j,k =1,\dots, N$ 
\begin{equation}\label{asss1}
	\int_{\mathbb{T}} \phi_j(x) \phi_k(x) dx = \delta_{j,k} 
\end{equation}
holds true, then $M(0)$ is the identity matrix. Provided that $\phi_j$ is continuously differentiable we obtain under assumption \eqref{asss1} that $M(t)$ is invertible for $t>0$ sufficiently small. Then,  we obtain local existence and uniqueness of solutions $\a$.  However, the following simple example shows that $M(t)$ is not necessarily invertible for all $t>0.$ Consider $\mathbb{T}=[-1,1]$, $N=2$, $2 \lambda=1$,  $\phi_1(x)=\sin(x \pi)$ and $\phi_2(x)=\sin(2 x \pi).$ Then, $M(0)=Id$ and $M(\frac12)=\begin{pmatrix} 0 & 0 \\ 0 & -1 \end{pmatrix}.$ 

\subsubsection{Case of Compactly Supported Translated Base Functions} 
Consider a compactly supported function $\phi_0: \mathbb{T} \to \R.$ For fixed $\Delta x=\frac{2}{N}$ sufficiently small, define the family of base functions 
\begin{equation}\label{basis einfach}
	\phi_j(x):=\phi_0\left( x - (j-1) \Delta x \right), \; j=1,\dots, N. 
\end{equation}
By definition of $\phi_j$ the base functions fulfill $\phi_j(x) = \phi_k\left(  x - (j-k)\Delta x \right).$ For $j=1,\dots, N, k=2,\dots,N$ we have 
\begin{align}\label{M definition}
M_{j,k}(t) = \int\limits_{ \mathbb{T}} \phi_j(x) \phi_k(x+2\lambda t) dx
=\int\limits_{ \mathbb{T}} \phi_j(x) \phi_{k-1}(x+2\lambda t-\Delta x) dx
=M_{j,k-1}\left( t-\frac{\Delta x}{2\lambda} \right) 
\end{align}
which implies that 
\begin{equation} M(t) = P^k \; M\left(t- k \frac{\Delta x}{2\lambda}
	\right), \;  k=1,\dots, \mbox{ and }  t \in \left[ k \frac{\Delta x}{2\lambda}, (k+1) \frac{\Delta x}{2\lambda} \right].  
	\end{equation} 
The permutation  matrix $P$ is  given by  
\begin{eqnarray}
	P_{i, mod(j+1,N)} = \delta_{i, j}, \; i,j=1,\dots N.
\end{eqnarray}
Hence, the family of matrices $M(t)$ for all $t\geq 0$ 
 is uniquely defined by 
$t \to M(t)$ for $t \in [0, \frac{\Delta x}{2\lambda}).$ 
Since $\phi_0$ is defined on $\mathbb{T}$ we obtain that $M(t)$ is a circulant matrix, i.e., for $i,j=1,\dots,N,$
\begin{equation}
	M_{i,j}(t)=M_{ mod(i+1,N), mod(j+1,N) }(t).
\end{equation}
The family of matrices $M(t)$   therefore uniquely defined by a family of vectors $\vec c=\vec c(t) \in \R^N$ with $c_j(t)=M_{1,j}(t)$ for $j=1,\dots,N$ and $t\geq 0.$ For circulant matrices the eigenvalues $\Lambda_M$ and $m=0,\dots,N-1$ are 
\begin{equation}\label{eigenvalues}
	\Lambda_m(t)=\sum\limits_{k=0}^{N-1} c_{k+1} (t)\; \exp\left( - 2\pi i \; \frac{ m k}N \right).
\end{equation}
The explicit eigenvalues \eqref{eigenvalues} determine possible $t$ such that $M(t)$ is not invertible. 
We illustrate this on two examples. Let $\phi_0(x)=\chi_{[- \frac{\Delta x}2, \frac{\Delta x}2]}(x)$ and let  $\phi_j$ be defined by equation \eqref{basis einfach}. Then, there exists $\Delta x>0$ and $N$ such that  the support $\Omega_j:=supp_x \phi_j(x)$ fulfills 
\begin{equation}\label{intersec}
	\Omega_i \cap \Omega_j = \emptyset, \; i \not= j, \mbox{ and } \cup_{j=1}^N \Omega_j = \mathbb{T}. 
\end{equation}
For this choice of $\{ \phi_j \}_{j=1}^N$ the vector $c(0)=(\Delta x,0, \dots, 0)^T$  and $c\left(\frac{\Delta x}{4 \lambda }\right)=\frac{\Delta x}2 (1,1,0,\dots,0)^T$. Hence, if $N$ is even, then  $\Lambda_m\left(\frac{\Delta x}{4 \lambda }\right)=0$ for $m=\frac{N}2$ and hence $M\left(\frac{\Delta x}{4 \lambda }\right)$ is not invertible. 
\par 
Similarly, if the support of $\phi_0$ is of size $\Delta x,$ i.e., $\phi_0(x)=\chi_{[- \Delta x, \Delta x] }(x)$, then $c(0)=(c_0,c_1,c_2,0,\dots, 0)^T$ with $c_0>c_j>0, j=2, 3$ and $M(0)$ is invertible. However, at time $t=\left(\frac{\Delta x}{4 \lambda }\right)$ and  $N \geq 4$ even,     we obtain $\Lambda_m=0$ for $m=\frac{N}4$.   

In the following we discuss properties of the matrix $M$ for the basis functions used in the numerical results later on. Hence, from now on we assume that $\phi_0$ is given by 

\begin{align}\label{phihut}
\phi_0(x)=\left\lbrace\begin{matrix}2x &x\in [0,\Delta x]\\4\Delta x-2x & x\in[\Delta x, 2\Delta x] \\ 0 & x \not \in [0, 2 \Delta x] \end{matrix}\right.
\end{align}
and $\phi_j$ for $j \geq 1$ are given by \eqref{basis einfach}.
There is an easy equivalence for when the  circulant matrix is singular.
  \begin{theorem}[\cite{chen2021}]\label{thm:circ}A circulant matrix made from the vector $c=[c_0,c_1,...,c_n]$ is singular if and only if $f(x)=\sum_{i=0}^{n-1}c_ix^i$ and $1-x^n$ have a common zero.
 \end{theorem}
 The matrix $M(t)$ resulting from the given basis function is non-singular almost everywhere. It is only non singular at discrete time instances and then there is only one zero eigenvalue:
 
 \begin{theorem}
 The matrix $M(t)$ given by \eqref{M definition} for $\phi$ given by \eqref{phihut} is non-singular on the interval $[0,\frac{1}{\lambda N}]$ as long as $t\neq t^*=\frac{1}{2\lambda N}$ and the nullspace at $t^*$ is only one-dimensional.
 \end{theorem}
 \begin{proof}
 In order to proof that the matrix is nonsingular we use Theorem \ref{thm:circ}. Our matrix is a circulant matrix composed of the vector $c=[c_1,...,c_N]$, where $c_j= \int\limits_{ \mathbb{T}} \phi_1(x) \phi_j(x+2\lambda t) dx$.
 In the given interval we have  $c_j=0$  except for  $c_1,c_2,c_3,c_N$. It is well known that the circulant matrix composed by $c_1,...c_N$ has up to sign the same determinant  as $c_N,c_1,...c_{N-1}$. Therefore we can consider  this matrix instead. Hence, the polynomial we are interested in is given by
 
 \[c_N+c_1 x+c_2 x^2+c_3 x^3\]. 
 Next, we  show that no root of unity is a zero of that polynomial except at time $t=t^*.$ 
 \end{proof}
 \begin{lemma}
 For $c_j(t)= \int\limits_{ \mathbb{T}} \phi_1(x) \phi_j(x+2\lambda t) dx$ the polynomial  $p(x,t)=c_N(t)+c_1(t) x+c_2(t) x^2+c_3(t) x^3$ has only a root of unity if $t=t^*.$
 \end{lemma}
 \begin{proof}
 Assume that $\omega$ is a root of unity and also a root of $p(x,t)$. Then $\omega$ is either complex, equal to $1$ or $-1$. However, $\omega=1$ can not be a root of $p$ as all $c_j$ are positive. If $\omega=-1$ is a root we have that $c_N(t)-c_1(t) +c_2(t)-c_3(t)=0$. It is straightforward by the definition of $c_j$ to show that $c_N(0)-c_1(0) +c_2(0)-c_3(0)<0$ and $c_N(\frac{1}{\lambda N})-c_1(\frac{1}{\lambda N}) +c_2(\frac{1}{\lambda N})-c_3(\frac{1}{\lambda N})>0$. Further, the derivative is positive in the given interval and therefore it has exactly one zero in this interval. This is at $t=t^*$. If $\omega$ is complex than also $\bar{\omega}$ has to be a root of $p(x,t)$ and then we obtain  

\[p(x,t)=(x-\omega)(x-\bar{\omega})(\alpha+\beta x)\]
 for some $\beta$ and $\alpha$. Comparing the coefficients we get that
 
 \begin{align}
 c_1&=\beta-2\alpha \Re(\omega)\\
c_2&=\alpha-2\beta \Re(\omega) \\
 c_3&=\beta\\
 c_N&=\alpha 
 \end{align}
 and from that we get that 
 \[\Re(\omega)=\frac{c_3-c_1}{2c_N}=\frac{c_N-c_2}{2c_3}.\]This fraction is always less or equal to $-1$ and therefore $\omega$ can only be $-1$ which has been treated before. 
 
\end{proof}
\subsubsection{Differential algebraic nature of the system \eqref{reform} }

As discussed in the previous section $M(t)$ could be singular for base functions fulfilling \eqref{intersec}. For the choices discussed above $M(t)$ is singular only
at a single point in time $t^*$ within the interval $\left[ 0, \frac{\Delta x}{2\lambda} \right]$, i.e. for the last example $t^* = \frac{\Delta x}{4\lambda}.$ 
 Furthermore,  it exist a vector $e$ such that $M(t^*)e=0$ and for all vectors $v$ orthogonal to that we have $M(t^*)v\neq 0$ unless $v=0$

 Let $V,W$ be the $N\times (N-1)$ dimensional orthogonal matrices and $f$ the vector orthogonal to $W$such that $W^TM(t)V$ is invertible and $f^TM(t)V=0$. Then,  decompose $\alpha^-$ into 
 \begin{equation}
 \alpha^-(t) =\alpha_0^-(t) e+V\bar{\alpha}^-(t). 
 \end{equation}
\renewcommand{\b}{ {\vec{\beta} }}
For $\b=( \alpha^+, \bar{\alpha}^-, \alpha^-_0 )$  problem \eqref{reform} reads 
\begin{align}\label{DAE2}
\begin{bmatrix} M(0) &-M(t)V &-M(t)e \\
W^TM(0) &W^TM(t)V & W^TM(t)e\\
f^TM(0) & 0 & f^TM(t)e\end{bmatrix}\frac{d}{dt}\b(t) =
\frac{1}\epsilon \begin{pmatrix} 0 \\  [M(0), M(t)] \b(t) -2 \tilde{F}(t,\b(t))  \end{pmatrix}. 
\end{align}
This system is not an ordinary differential equation at $t=t^*$, since  $M(t^*)e=0$. The resulting system is a semi--explicit differential algebraic equation. We introduce a small parameter $\rho>0$ and regularize equation \eqref{DAE2} by 

%

\begin{align}\label{DAE3}
\begin{bmatrix} M(0) &-M(t)V &-M(t)e\\
W^TM(0) &W^TM(t)V & W^TM(t)e\\
f^TM(0) & 0 & f^TM(t)e+\rho\end{bmatrix}\frac{d}{dt}\b =
\frac{1}\epsilon \begin{pmatrix} 0 \\  [M(0), M(t)] \b(t) -2 \tilde{F}(t,\b(t))  \end{pmatrix}. 
\end{align}

or in terms of $\a$ we have 

\begin{align}\label{DAE}
\begin{bmatrix} M(0) &-M(t)\\
M(0) & M(t)+\rho  fe^T\end{bmatrix}\frac{d}{dt}\a = 
\frac{1}\epsilon \begin{pmatrix} 0 \\  [M(0), M(t)] \a(t) -2 \tilde{F}(t,\a(t))  \end{pmatrix}. 
\end{align}

Fpr $\rho>0$ the matrix is invertible and it inverse is given by 
\begin{equation}
\label{eq:Einv}
\begin{bmatrix}M(0)^{-1}(I+M(t)(2M(t)+\rho  f e^T )^{-1}) & M(0)^{-1}M(t)(2M(t)+\rho  f e^T)^{-1}\\
-(2M(t)+\rho  f e^T)^{-1}&  (2M(t)+\rho  f e^T)^{-1} \end{bmatrix}.
\end{equation}

%
%
%
\subsection{Temporal Discretization  and Model Order Reduction }

Fix a positive parameter $\rho>0$ and consider  system \eqref{DAE} subject to initial conditions \eqref{ic red}. 
Consider a temporal grid $t^n=\Delta t \; n$ for $n=0,\dots,$ where for simplicity we consider an equi-distant grid in time.  Denote by $\alpha^\pm_n = \alpha^\pm(t^n).$ Furthermore, denote by $$N(t):=M(t) + \rho f e^T.$$  We rewrite \eqref{DAE} as 
\begin{align}
	\frac{d}{dt} (M(0)\alpha^+-M(t)\alpha^-)&=-\dot{M}(t)\alpha^-,\\
	\frac{d}{dt} (M(0)\alpha^++N(t)\alpha^-)&=\dot{M}(t)\alpha^-- \frac{2}{\epsilon} \left(
	\frac12 (M(0)\alpha^++M(t) \alpha^-) - \tilde{F}(t,\alpha^\pm(t)) 
	\right), \label{temp2}
\end{align}
where $\tilde{F}$ is given by equation \eqref{tildeF} and $\dot{M}_{i,j}(t)=\int_{\mathbb{T}} \phi_j(x) \phi_k'(x+2\lambda  t ) 2 \lambda dx.$ An implicit discretization of \eqref{temp2} is preferable to resolve small scales of $\epsilon.$ Since the term $\tilde{F}$ is an integral term in both $\alpha^+$ and $\alpha^-$  a fully implicit discretiaztion is computationally too costly. We therefore proceed using a semi-implicit discretization, i.e., 

\begin{align}
	(M_0\alpha^+_{n+1}-M_{n+1}\alpha^-_{n+1})-(M_0\alpha^+_{n}-M_{n}\alpha^-_{n})=&-\Delta t \dot{M}_n\alpha^-_n\\
	(M_0\alpha^+_{n+1}+N_{n+1}\alpha^-_{n+1})-(M_0\alpha^+_{n}+N_{n}\alpha^-_{n})=&\Delta t \dot{M}_n\alpha^-_n -\frac{ \Delta t  }{\epsilon} \left(
	 M_0\alpha^+_{n+1}+M_{n+1} \alpha^-_{n+1} - \tilde{F}_n
	\right), \\ 
	\tilde{F}_n =& 2 \tilde{F}(t^n,\alpha^\pm_n),
\end{align}
leading to the following system
\begin{align}\label{temp3}
	\begin{bmatrix} M_0 &-M_{n+1}\\ M_0 & N_{n+1}\end{bmatrix}
	\begin{bmatrix}\alpha^+_{n+1}\\ \alpha^-_{n+1}\end{bmatrix}=
	\begin{bmatrix}
		M_0\alpha^+_n-M_n\alpha_n^--\Delta t \dot{M}_n\alpha^-_n\\
		\frac{\epsilon}{\epsilon+\Delta t} \left(M_0\alpha^+_n+N_n\alpha_n^-+\Delta t \dot{M}_n\alpha^-_n \right) +  \frac{\Delta  t}{\epsilon+\Delta t } \tilde{F}_n
	\end{bmatrix}
\end{align}
As in equation \eqref{reform} the left-hand side of equation \eqref{temp3} consists of a $2\times 2$ block matrix 
$$R_{n+1}:=\begin{bmatrix} M_0 &-M_{n+1}\\ M_0 &N_{n+1}\end{bmatrix}$$ which is invertible provided that $M(0)$ is invertible and 
 $\rho$ is non--negative. In this case its inverse is given by equation \eqref{eq:Einv} evaluated at $t=t^n.$  
  Furthermore, $\tilde{F}_n$ and $\dot{M}_n=\dot{M}(t^n)$ needs to be discretized using a numerical quadrature formula of sufficient high--order.  
  Note that  the previous formulation can be formally evaluated for all values of $\epsilon$ (even $\epsilon=0$). However, since $\tilde{F}_n=\tilde{F}(t^n, \alpha^+_n, \alpha^-_n)$ the previous scheme requires a time step restriction of the type 
\begin{align}\label{cfl}
	\Delta t \leq C \epsilon 
\end{align}
for some constant $C$ to be stable. Clearly, this leads to small steps for sufficiently small $\epsilon.$ The only way to circumvent this restriction is to discretize $\tilde{F}$ implicit.  Since our focus is on the model order reduction for the system \eqref{temp3} we leave the efficient computation of the fully implicit scheme for future investigation. For sake of completeness we also state the alternative fully explicit discretization as 

\begin{align}
(M_0\alpha^+_{n+1}-M_{n+1}\alpha^-_{n+1})-(M_0\alpha^+_{n}-M_{n}\alpha^-_{n})=&-\Delta t \dot{M}_n\alpha^-_n\\
(M_0\alpha^+_{n+1}+N_{n+1}\alpha^-_{n+1})-(M_0\alpha^+_{n}+N_{n}\alpha^-_{n})=&\Delta t \dot{M}_n\alpha^-_n - \frac{2 \Delta t}{\epsilon}(
\frac12 (M_0\alpha^+_{n}+M_{n} \alpha^-_{n}) - \tilde{F}_n
) \label{explicit disc} \\
\tilde{F}_n =& \tilde{F}(t^n,\alpha^\pm_n).
\end{align}
The same restriction \eqref{cfl} applies for this discretization. 
\par 
Note that the original scheme  \cite{Jin1995b} does {\bf not } require a time step restriction of the order of $\epsilon.$ The reason being that in the case of a finite volume scheme the implicit discretization of the source term can be evaluated analytically. As outlined above, the basis functions $\phi$ however couple the coefficients and this coupling prevents an analytically evaluation.

\subsection{Projection based Model Order Reduction for system \eqref{eq:ode1}-\eqref{eq:ode2}}

The previous formulation \eqref{temp3} is amendable for model order reduction. Hence, we approximate 
 $\alpha^{\pm}(t)\in\mathbb{R}^N$ within a lower dimensional linear subspace of $\mathbb{R}^N$,
meaning there exist $V_+$ and $V_-$ such that $\alpha_{\pm}(t)\approx V_{\pm}V_{\pm}^T\alpha^{\pm}(t)$ and therefore
 an $\hat{\alpha}^{\pm} (t) $ exist such that  $\alpha^{\pm}(t)\approx V_{\pm}\hat{\alpha}^{\pm}(t)$.

 Using this approximation in the ODE we get the following 
\begin{align}
M(0)V_+\dot{\hat{\alpha}}^+(t)  - M(t)V_-\dot{\hat{\alpha}}^-(t) &=0\label{eq:aode1}\\
M(0)V_+\dot{\hat{\alpha}}^+(t) +M(t) V_-\dot{\hat{\alpha}}^-(t) &= - \frac{2}{\epsilon} \left( 
\frac12 \left(M(0)V_+\hat{\alpha}^+(t) +M(t) V_-\hat{\alpha}^-(t) \right) - \tilde{F}(t,V_{\pm} \hat{\alpha}^\pm(t) )
\right) \label{eq:aode2} , 
\end{align}
which is then projected to get a system of ordinary differential equation in a lower dimension. We use a Galerkin projection for simplicity. However we first solve for $V_{\pm}\hat{\alpha}^\pm$  and then multiply the equation by the transpose of the projections matrices $V_\pm$: 

\begin{align}
\dot{\hat{\alpha}}^+(t) &=-\frac{1}{2}V_+^TM^{-1}(0) \frac{2}{\epsilon} \left( 
\frac12 \left(M(0)V_+\hat{\alpha}^+(t) +M(t) V_-\hat{\alpha}^-(t) \right) - \tilde{F}(t,V_{\pm} \hat{\alpha}^\pm(t) )
\right)\label{eq:rode1}\\
\dot{\hat{\alpha}}^-(t) &= - V_-^T M^{-1}(t)\frac{1}{\epsilon} \left( 
\frac12 \left(M(0)V_+\hat{\alpha}^+(t) +M(t) V_-\hat{\alpha}^-(t) \right) - \tilde{F}(t,V_{\pm} \hat{\alpha}^\pm(t) )
\right) \label{eq:rode2} , 
\end{align}
As above if $M(t)$ is not invertible we replace it by $N(t)$.

In order to gain computation speed solving equation \eqref{eq:rode1} and \eqref{eq:rode2} over the full system \eqref{DAE} we need to make sure that the right hand sides are evaluated fast and do not need the computation of vectors of the full size. This is for arbitrary nonlinear flux function and arbitrary basis functions $\phi$ not a trivial problem. However this paper is concerned with the proof of concept of the general method, namely the fact that the solution of $\alpha$ in $\mathbb{R}^{2N}$ lives in a low-dimensional space and this fact can be exploited to create a reduced model with standard methods for $\alpha$.

\section{Computational Results}

The theoretical findings are exemplified on a series of linear and nonlinear numerical examples. All computational results are obtained on  torus  $\mathbb{T}=[-1,1].$ The matrix $M(t)$ defined by equation \eqref{mass matrix}  and the $j$ the component of the right-hand side  $\tilde{F}$ are given by  
 \begin{align}
M_{j,k}(t) = \int\limits_{-1+\lambda t }^{1+\lambda t } \phi_j(x-\lambda t) \phi_k(x+ \lambda t) dx, \; &
\tilde{F}_j(t,\a) = \int\limits_{-1+\lambda t }^{1+\lambda t }    \phi_j(x-\lambda t ) f(   
\tilde{u}(t,x)   )dx, 
\end{align}
where $\tilde{u}$ is given by equation \eqref{tildeu}. As base function we choose  compactly supported, piecewise linear  functions fulfilling  property  \eqref{intersec}. We divide the torus in cells $[j-1,j]\Delta x$ where  for fixed $N$ we set $\Delta x=\frac{2}{N}$ and $j=1,\dots N$. Then, the set of  base functions $\{ \phi_j : j=1,\dots, N\}$ are defined by 
\begin{align}
\phi_j(x)=\left \{\begin{matrix}
x-(j-1)\Delta x& x \in[j-1,j]\Delta x\\
(j+1)\Delta x-x& x\in[j-1,j]\Delta x \\
 0 & \mbox{else} 
\end{matrix}\right. 
\end{align}
As shown above the matrix  $M(0)$ is invertible for the previous choice of $\phi_j.$  
Further, $M(t)$ is invertible for $t>0$, except at the discrete time $t_\ell=\frac{2\ell+1}{2\lambda N}$. The number of  time steps is denoted by $N_{\Delta t}.$

\medskip 
 
The further parameters are set as follows: 
\begin{align}
	\rho=\epsilon, ; \Delta t= \frac12 \epsilon \mbox{ and } \; T=N_{\Delta t} \Delta t.
\end{align}

Since the basis function is only nonzero on a small interval we use this in the numerical implementation with MATLAB\,\textsuperscript{\tiny\textregistered} built in function \texttt{integral}.
When we compute the matrix $M$ at any  time $t=t_n$ where we use the fact that our basis functions $\phi$ are  shifted. This implies that we  have to compute only a single  row of the matrix as the matrix is a circulant matrix. To be even more precise, as only four of the values are potentially nonzero, we only have to compute those. Besides $M$ and $N$ we have to also compute $\tilde{F}$ at time $t_n$. As $\alpha_n^\pm$ are given we can define the function $\tilde{u}(t_n,x)$ and $\tilde{F}$ compute via a quadrature rule which we do by using the build in MATLAB\,\textsuperscript{\tiny\textregistered} function \texttt{integral}.  Once we have the initial values for $\alpha^\pm$ and the possibility to evaluate $M$ and $\tilde{F}$ we use \eqref{temp3} to compute further timesteps of $\alpha^\pm$. Since we are only interested  in the qualitative behaviour,   we do not discuss the possibilities for improving this numerical computation.
\medskip
In the numerical results we will first show that the approach using translated base functions yields qualitative and quantitative correct solutions in the case of linear transport with and without nonlinear source terms, showing that this discretization produces feasible solution. We then show that a linear subspace in the solution space of $\alpha^\pm$  produces correct results and with that the reducability of the ordinary differential equation in $\alpha$.  Secondly, we show that also for nonlinear transport the proposed method yields a good qualitative and quantitative agreement with standard results by  finite--volume methods. The later however are not amendable for model order reduction. In the case of strong shocks the reduction in  dimension of the reduced model order system is however not as significant as in the linear case. However, the computed reduced order system is able to correctly reproduce solutions to different initial data. This example shows that the chosen formulation is  amendable for model order reduction even in the nonlinear case and in the case of discontinuous solutions.

\subsection{Linear Transport with Nonlinear Source}

In order to validate the Ansatz \eqref{Ansatz} we present  numerical results for  linear transport equation with nonlinear source term: 
\begin{align}
	\label{simple}
	\partial_t w(t,x) +\lambda  \partial_x  w(t,x) =\gamma w^2(t,x)+\delta w \\ 
	w(0,x)=w_0(x) \not = 0 \label{simple_id}
\end{align}
The equation \eqref{simple} contains three parameters $\lambda \not =0,$ and $\gamma, \delta \in \R.$
The case $\gamma=\delta=0$ corresponds to a linear transport equation.  On the full space $x \in \R$ the explicit solution to 
equation \eqref{simple} and \eqref{simple_id} is given by 
\begin{equation} \label{simple expl} 
	w(t,x)=\frac{1}{\frac{e^{-\delta t}}{w_0(x-\lambda t)}-\frac{\gamma}{\delta}(1-e^{-\delta  t})}.
\end{equation}
for $t$ sufficiently small such that \eqref{simple expl} is well--defined. Due to the finite speed of propagation a numerical 
 comparison of approximation errors with the exact solution is  possible  provided that $supp w_0(x) \subset\subset \mathbb{T}$. In this
 case the exact solution $w(t,x)$ is given  by equation \eqref{simple expl} for  $x \in  {\bf S}(t)$ where ${\bf S}(t):= \{x:  x- \lambda t \in supp w_0 \}$ and $w=0$ zero else. The exact solution $w(t,x)$ is defined for any $t$ such that ${\bf S}(t) \subset \mathbb{T}.$  In the case of the linear equation our Ansatz reduces to 

\begin{equation}\label{eq:walpha}w(t,x)=\sum_{i=1}^{N} \alpha_i(t)\phi_i(x-\lambda t) 
\end{equation}
\begin{figure}[htb]
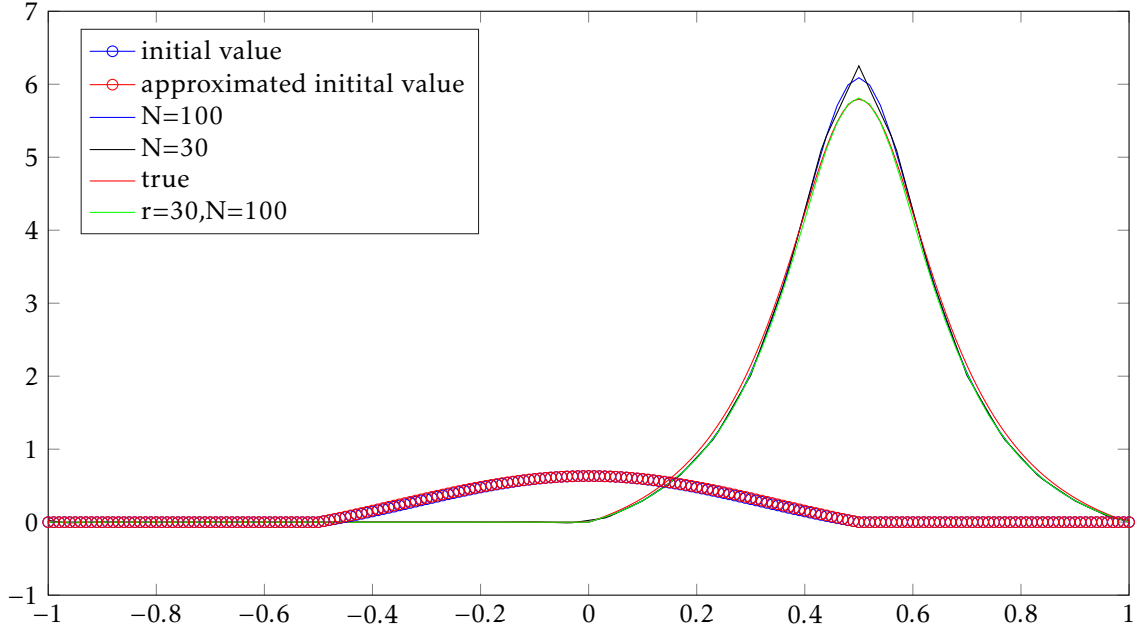

	\include{test00}
\caption{\label{fig:test0}
Simulation result for a linear transport equation with a nonlinear right hand side with a solution ansatz as in equation \eqref{eq:walpha} for two different values of $N$ and also a simulated reduced system arising from the larger system and a reduced order equivalent to the size of the smaller one.}
\end{figure}
 In the numerical test shown in Figure \ref{fig:test0} we choose  
$\gamma=2,\delta =1$ and simulate until $T=0.5$. We compute the solution for $N=100$ and a reduced model  projected on a $r=30$ dimensional subspace. We compare the analytical solution, the solution on the subspace with $N=100$ and the solution with $N=30$ modes. The initial value is given by $w_0(x)=e^{-4x^2}-e^{-1}$ on $(-1/2,1/2)$ and zero otherwise. Its approximation on the subspace is also shown. As seen in Figure \ref{fig:test0}  we observe very good agreement between the reduced basis approximation and the analytical solution. Clearly, higher-dimensional subspaces provide better agreement than lower dimensional ones. This example indicates that the use of translate base functions leads to qualitative and quantitative correct results in the linear case. 

\subsection{Relaxation Approach for a Linear Problems}

Consider the same linear flux as in the previous example. Here, we apply the relaxation formulation with $\lambda=a$ and $\epsilon=10^{-3}$ to the linear problem. Clearly, this is not necessary in order to solve the linear problem but the numerical result following illustrates that no additional numerical approximation error appears. 
\medskip 
The initial condition is $$u_0=\sin(\pi x)$$  and the analytical solution is given by $u(t,x)=u_0(x-t)$. In Figure \ref{fig:test_lin} we show  initial condition and analytical solution at  final time $T=1$.   Figure \ref{fig:test_lin} shows that using $N=40$ basis functions  the numerical solution is indistinguishable from the analytical solution. Further, we observe that in this particular example a reduced system of dimension two can already capture the complete behavior due to the fact that there is only linear transport. For sake of completeness we also show the result with only a single base function that is equal to zero. The test case only contains smooth data and solution and as expected the a low dimensional reduced base formulation recovers the behavior well.

\begin{figure}[hbt]
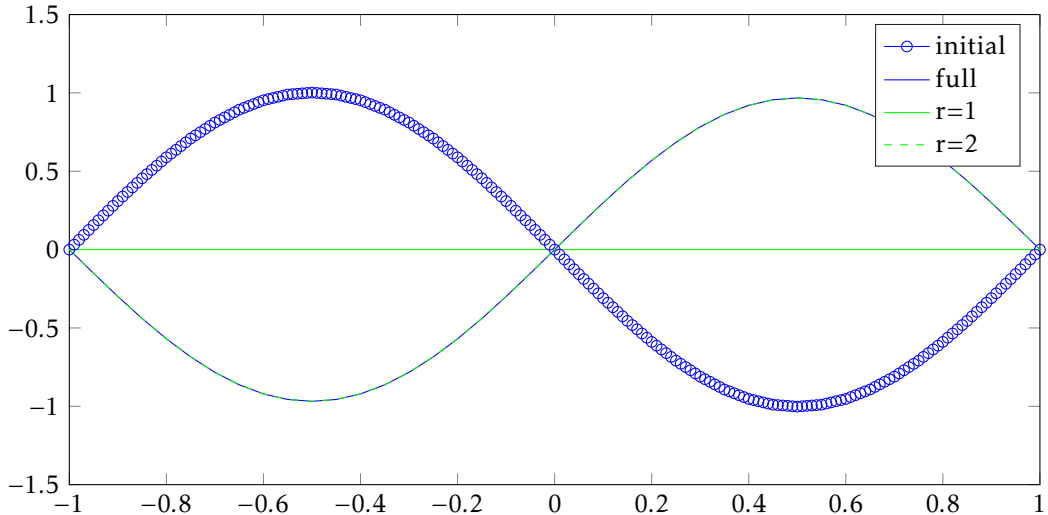

	\include{test_1_plot}

	\caption{  	\label{fig:test_lin}
		Relaxation formulation applied to a linear flux $f(u)=u$ and smooth initial data. Shown are the numerical solution at time $T=1$ with $N=40$ piecewise defined basis function and  $r=2$ and $r=1$ reduced basis functions, respectively. 
}
\end{figure}

\subsection{ Burgers Equation and Approximation of Shock Solution}

We consider the relaxation formulation for Burgers equation, i.e., the flux is given by  $f(u)=\frac12 u^2$. Smooth periodic initial data 
$$u_0(x)=\frac12+\sin(\pi x)$$ 
on $\mathbb{T}$ is considered. It is known that at time $T=1$ a shock is formed due to the nonlinear transport. In Figure \ref{fig:BurgerNstudy} we show the quality of the proposed approximation for different numbers of base functions $N.$ We choose $\lambda$ larger than the norm of the initial data, i.e., 

$$ \lambda = 2$$ and 
$$\epsilon=10^{-3}$$ 
for this test. In the subfigures of Figure  \ref{fig:BurgerNstudy} initial data and the solution at terminal time $T=1$ is shown. We observe that for $N$ sufficiently large the expected shock is recovered in detail. For small $N$ we observe a Gibb's phenomena due to the strong discontinuity of the underlying solution. To compare the solution we also included a figure showing the result of a second--order finite volume scheme applied to the {\bf same} relaxation formulation. In particular, we observe that the size of the jump discontinuity is the same for the proposed approximation and the finite--volume scheme. The later is taken from reference \cite{Jin1995b} and the spatial discretization is given by $\Delta x=1/320.$ Since $\epsilon>0$ we observe in all simulations a slight decay of the maxima and minima over time. For smaller values of $\epsilon$ the decay of the extreme values is expected to be smaller. However, the time step of the proposed method scales with $\epsilon$ and this leads to inefficiencies in the numerical scheme. Compared to the method \cite{Jin1995b} we can not resolve in the regime $\Delta t>\epsilon.$ 

\begin{figure}[htb]
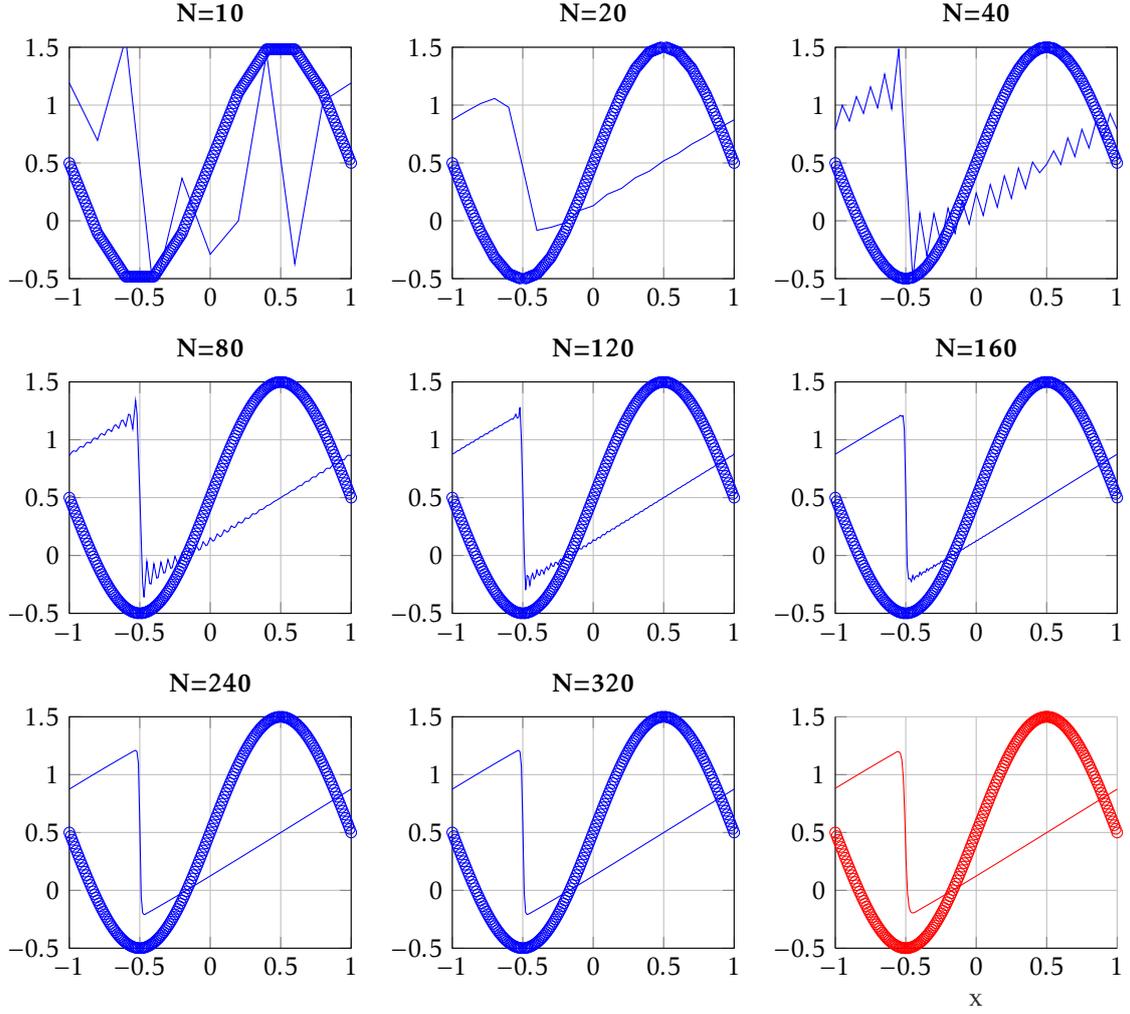

\include{test_2_plot}
\caption{\label{fig:BurgerNstudy}
	Relaxation approximation to a solution to Burgers equation with smooth initial data. At time $T=1$ a shock develops that is captured by the proposed approximation for $N$ sufficiently large. In red a comparison with a first--order finite volume scheme with $N=320$ discretization points in space.  }
\end{figure}

In Figure \ref{fig:sing}(the red  curve) we observe decay of singular values in the solution  such that we can derive efficient model order reduction formulations  in the classical sense. As expected the decay is not as significant as in elliptic or parabolic problems.

\subsection{Model Order Reduction Strong Shock}

As a second example with non--smooth data we consider  Burgers' equations and initial data of the type   
$$u_0(x)= a(   \chi_{0, 1/2 }(x) - 1 ),$$  
for a parameter $a>0.$ The value of $a$ controls the size of the jump discontinuity. The solution $u$ to Burgers equation and the given initial data consists of a shock wave followed by a rarefaction. The later wave is a linear function. The parameter $a$ also controls the speed of propagation of the shock wave due to the Rankine-Hugenoit condition:  the speed is $s=-a\frac14.$ 
For the numerical test we set $\epsilon=10^{-3}$ and  $\Delta t=\frac{\epsilon}{10}$.   In Figure \ref{fig:burgermor}  the initial condition and its approximation  with a discretization of $N=160$ are shown in the left part of the figure. Small oscillations due to the strong discontinuity  are visible. On the right we show the solution for two  reduced model order approximations  as well as the full model and a reference solution. The later is computed as in the previous section using a second--order finite volume scheme with $\Delta x=\frac{1}{160}.$ 

   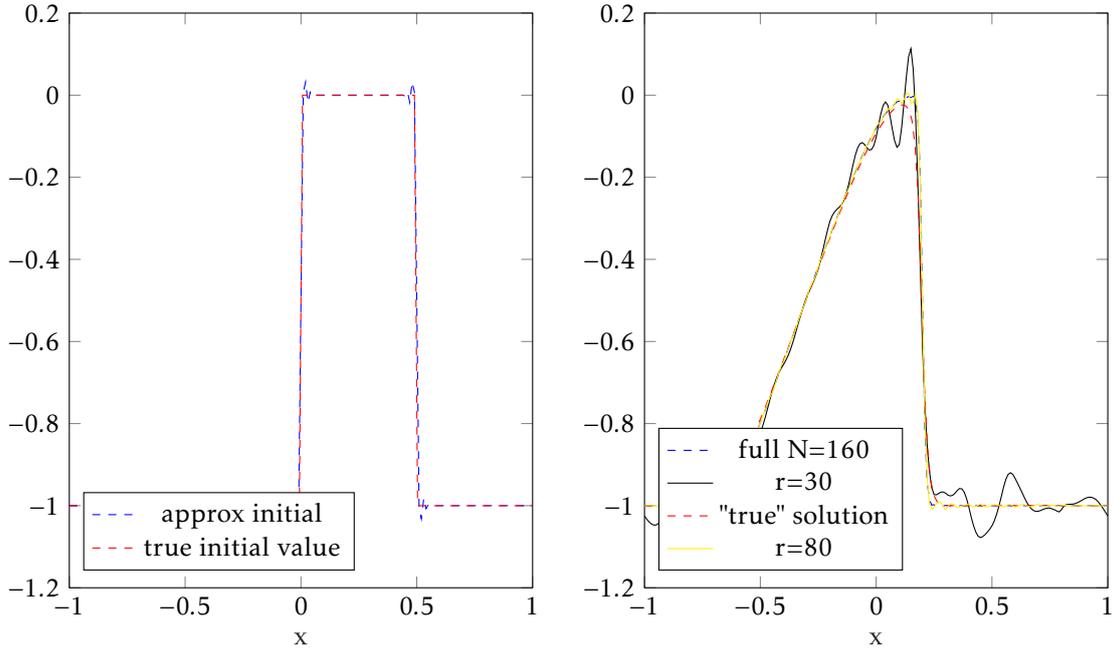
\begin{figure}[hbt]
 \begin{subfigure}{0.5\textwidth}
%
%
\definecolor{mycolor1}{rgb}{1.00000,0.00000,1.00000}%
\definecolor{mycolor2}{rgb}{1.00000,1.00000,0.00000}%
\begin{tikzpicture}

\begin{axis}[%
width=2.4in,
height=3in,
at={(1.504in,0.591in)},
scale only axis,
xmin=-1,
xmax=1,
xlabel style={font=\color{white!15!black}},
xlabel={x},
ymin=-1.2,
ymax=0.2,
axis background/.style={fill=white},
title style={font=\bfseries},
legend pos = south west,
]
\addplot [color=blue, dashed]
  table[row sep=crcr]{%
-1	-0.999999999999994\\
-0.99	-1\\
-0.98	-0.999999999999997\\
-0.97	-1.00000000000001\\
-0.96	-1.00000000000002\\
-0.95	-0.999999999999992\\
-0.94	-0.999999999999999\\
-0.93	-1\\
-0.92	-0.999999999999998\\
-0.91	-1\\
-0.9	-1.00000000000003\\
-0.89	-0.999999999999999\\
-0.88	-0.999999999999997\\
-0.87	-1\\
-0.86	-1\\
-0.85	-0.999999999999993\\
-0.84	-1.00000000000002\\
-0.83	-1.00000000000001\\
-0.82	-0.999999999999996\\
-0.81	-1\\
-0.8	-1\\
-0.79	-0.999999999999993\\
-0.78	-1.00000000000001\\
-0.77	-1.00000000000001\\
-0.76	-0.999999999999993\\
-0.75	-1\\
-0.74	-1\\
-0.73	-0.999999999999996\\
-0.72	-1.00000000000001\\
-0.71	-1.00000000000002\\
-0.7	-0.999999999999991\\
-0.69	-1\\
-0.68	-1\\
-0.67	-0.999999999999998\\
-0.66	-0.999999999999999\\
-0.65	-1.00000000000003\\
-0.64	-0.999999999999999\\
-0.63	-0.999999999999998\\
-0.62	-1\\
-0.61	-1\\
-0.6	-0.999999999999991\\
-0.59	-1.00000000000002\\
-0.58	-1.00000000000001\\
-0.57	-0.999999999999996\\
-0.56	-1\\
-0.55	-1\\
-0.54	-0.999999999999993\\
-0.53	-1.00000000000001\\
-0.52	-1.00000000000001\\
-0.51	-0.999999999999993\\
-0.5	-1\\
-0.49	-1\\
-0.48	-0.999999999999996\\
-0.47	-1.00000000000001\\
-0.46	-1.00000000000002\\
-0.45	-0.999999999999991\\
-0.44	-1\\
-0.43	-1\\
-0.42	-0.999999999999998\\
-0.41	-0.999999999999999\\
-0.4	-1.00000000000003\\
-0.39	-0.999999999999999\\
-0.38	-0.999999999999998\\
-0.37	-1\\
-0.36	-1\\
-0.35	-0.999999999999991\\
-0.34	-1.00000000000002\\
-0.33	-1.00000000000001\\
-0.32	-0.999999999999996\\
-0.31	-1\\
-0.3	-0.999999999999993\\
-0.29	-1.00000000000002\\
-0.28	-0.99999999999995\\
-0.27	-1.00000000000013\\
-0.26	-1.00000000000002\\
-0.25	-0.999999999998182\\
-0.24	-1.00000000000507\\
-0.23	-0.999999999987504\\
-0.22	-1.00000000002264\\
-0.21	-1.00000000000509\\
-0.2	-0.999999999646963\\
-0.19	-1.00000000098341\\
-0.18	-0.999999997576782\\
-0.17	-1.00000000439004\\
-0.16	-1.00000000098341\\
-0.15	-0.999999931514404\\
-0.14	-1.00000019077634\\
-0.13	-0.999999529907697\\
-0.12	-1.00000085164499\\
-0.11	-1.00000019077635\\
-0.1	-0.999986714141449\\
-0.09	-1.00003700962759\\
-0.08	-0.999908804516933\\
-0.07	-1.00016521473821\\
-0.0599999999999999	-1.00003700962757\\
-0.0499999999999999	-0.997422611928567\\
-0.04	-1.00717967697244\\
-0.03	-0.982308546376036\\
-0.02	-1.03205080756889\\
-0.01	-1.00717967697244\\
0	-0.500000000000006\\
0.01	0.00717967697244942\\
0.02	0.0320508075688775\\
0.03	-0.0176914536239792\\
0.04	0.0071796769724491\\
0.0499999999999999	-0.00257738807143573\\
0.0599999999999999	3.70096275710883e-05\\
0.07	0.00016521473821582\\
0.08	-9.11954830736072e-05\\
0.09	3.7009627571105e-05\\
0.1	-1.32858585415591e-05\\
0.11	1.90776345260756e-07\\
0.12	8.51644991003949e-07\\
0.13	-4.70092300482371e-07\\
0.14	1.90776345260786e-07\\
0.15	-6.84856266959482e-08\\
0.16	9.83409487869983e-10\\
0.17	4.39003894992058e-09\\
0.18	-2.42321997418094e-09\\
0.19	9.83409487869762e-10\\
0.2	-3.53037454934731e-10\\
0.21	5.09538594750984e-12\\
0.22	2.25652806246948e-11\\
0.23	-1.23745087296713e-11\\
0.24	5.09538594751163e-12\\
0.25	-3.63956139107892e-12\\
0.26	5.09538594750949e-12\\
0.27	-1.23745087296662e-11\\
0.28	2.25652806246862e-11\\
0.29	5.09538594750596e-12\\
0.3	-3.53037454934608e-10\\
0.31	9.83409487869415e-10\\
0.32	-2.42321997418013e-09\\
0.33	4.39003894991911e-09\\
0.34	9.8340948786951e-10\\
0.35	-6.84856266959296e-08\\
0.36	1.90776345260739e-07\\
0.37	-4.70092300482266e-07\\
0.38	8.51644991003784e-07\\
0.39	1.90776345260681e-07\\
0.4	-1.32858585415567e-05\\
0.41	3.70096275710989e-05\\
0.42	-9.11954830735951e-05\\
0.43	0.000165214738215796\\
0.44	3.70096275711023e-05\\
0.45	-0.00257738807143553\\
0.46	0.00717967697244856\\
0.47	-0.0176914536239783\\
0.48	0.0320508075688772\\
0.49	0.00717967697244643\\
0.5	-0.500000000000002\\
0.51	-1.00717967697244\\
0.52	-1.03205080756888\\
0.53	-0.982308546376027\\
0.54	-1.00717967697244\\
0.55	-0.997422611928579\\
0.56	-1.00003700962757\\
0.57	-1.00016521473821\\
0.58	-0.999908804516929\\
0.59	-1.00003700962758\\
0.6	-0.999986714141448\\
0.61	-1.00000019077636\\
0.62	-1.00000085164499\\
0.63	-0.999999529907695\\
0.64	-1.00000019077634\\
0.65	-0.999999931514388\\
0.66	-1.0000000009834\\
0.67	-1.00000000439004\\
0.68	-0.999999997576787\\
0.69	-1.0000000009834\\
0.7	-0.999999999646967\\
0.71	-1.00000000000508\\
0.72	-1.00000000002263\\
0.73	-0.999999999987509\\
0.74	-1.00000000000508\\
0.75	-0.999999999998175\\
0.76	-1.00000000000002\\
0.77	-1.00000000000012\\
0.78	-0.99999999999994\\
0.79	-1.00000000000002\\
0.8	-1.00000000000001\\
0.81	-0.999999999999999\\
0.82	-0.999999999999996\\
0.83	-1\\
0.84	-1.00000000000001\\
0.85	-0.99999999999999\\
0.86	-1.00000000000001\\
0.87	-1\\
0.88	-0.999999999999995\\
0.89	-0.999999999999999\\
0.9	-1.00000000000001\\
0.91	-0.999999999999995\\
0.92	-1.00000000000001\\
0.93	-1.00000000000001\\
0.94	-0.999999999999995\\
0.95	-0.999999999999995\\
0.96	-1.00000000000001\\
0.97	-1\\
0.98	-1\\
0.99	-1.00000000000001\\
1	-0.999999999999994\\
};
\addlegendentry{approx initial}

\addplot [color=red, dashed]
  table[row sep=crcr]{%
-1	-1\\
-0.986577181208054	-1\\
-0.973154362416107	-1\\
-0.959731543624161	-1\\
-0.946308724832215	-1\\
-0.932885906040268	-1\\
-0.919463087248322	-1\\
-0.906040268456376	-1\\
-0.89261744966443	-1\\
-0.879194630872483	-1\\
-0.865771812080537	-1\\
-0.852348993288591	-1\\
-0.838926174496644	-1\\
-0.825503355704698	-1\\
-0.812080536912752	-1\\
-0.798657718120805	-1\\
-0.785234899328859	-1\\
-0.771812080536913	-1\\
-0.758389261744966	-1\\
-0.74496644295302	-1\\
-0.731543624161074	-1\\
-0.718120805369127	-1\\
-0.704697986577181	-1\\
-0.691275167785235	-1\\
-0.677852348993289	-1\\
-0.664429530201342	-1\\
-0.651006711409396	-1\\
-0.63758389261745	-1\\
-0.624161073825503	-1\\
-0.610738255033557	-1\\
-0.597315436241611	-1\\
-0.583892617449664	-1\\
-0.570469798657718	-1\\
-0.557046979865772	-1\\
-0.543624161073825	-1\\
-0.530201342281879	-1\\
-0.516778523489933	-1\\
-0.503355704697987	-1\\
-0.48993288590604	-1\\
-0.476510067114094	-1\\
-0.463087248322148	-1\\
-0.449664429530201	-1\\
-0.436241610738255	-1\\
-0.422818791946309	-1\\
-0.409395973154362	-1\\
-0.395973154362416	-1\\
-0.38255033557047	-1\\
-0.369127516778524	-1\\
-0.355704697986577	-1\\
-0.342281879194631	-1\\
-0.328859060402685	-1\\
-0.315436241610738	-1\\
-0.302013422818792	-1\\
-0.288590604026846	-1\\
-0.275167785234899	-1\\
-0.261744966442953	-1\\
-0.248322147651007	-1\\
-0.23489932885906	-1\\
-0.221476510067114	-1\\
-0.208053691275168	-1\\
-0.194630872483222	-1\\
-0.181208053691275	-1\\
-0.167785234899329	-1\\
-0.154362416107383	-1\\
-0.140939597315436	-1\\
-0.12751677852349	-1\\
-0.114093959731544	-1\\
-0.100671140939597	-1\\
-0.087248322147651	-1\\
-0.0738255033557047	-1\\
-0.0604026845637584	-1\\
-0.0469798657718121	-1\\
-0.0335570469798657	-1\\
-0.0201342281879194	-1\\
-0.00671140939597314	-1\\
0.00671140939597326	0\\
0.0201342281879195	0\\
0.0335570469798658	0\\
0.0469798657718121	0\\
0.0604026845637584	0\\
0.0738255033557047	0\\
0.087248322147651	0\\
0.100671140939597	0\\
0.114093959731544	0\\
0.12751677852349	0\\
0.140939597315436	0\\
0.154362416107383	0\\
0.167785234899329	0\\
0.181208053691275	0\\
0.194630872483222	0\\
0.208053691275168	0\\
0.221476510067114	0\\
0.23489932885906	0\\
0.248322147651007	0\\
0.261744966442953	0\\
0.275167785234899	0\\
0.288590604026846	0\\
0.302013422818792	0\\
0.315436241610738	0\\
0.328859060402685	0\\
0.342281879194631	0\\
0.355704697986577	0\\
0.369127516778524	0\\
0.38255033557047	0\\
0.395973154362416	0\\
0.409395973154362	0\\
0.422818791946309	0\\
0.436241610738255	0\\
0.449664429530201	0\\
0.463087248322148	0\\
0.476510067114094	0\\
0.48993288590604	0\\
0.503355704697987	-1\\
0.516778523489933	-1\\
0.530201342281879	-1\\
0.543624161073825	-1\\
0.557046979865772	-1\\
0.570469798657718	-1\\
0.583892617449664	-1\\
0.597315436241611	-1\\
0.610738255033557	-1\\
0.624161073825503	-1\\
0.63758389261745	-1\\
0.651006711409396	-1\\
0.664429530201342	-1\\
0.677852348993289	-1\\
0.691275167785235	-1\\
0.704697986577181	-1\\
0.718120805369127	-1\\
0.731543624161074	-1\\
0.74496644295302	-1\\
0.758389261744966	-1\\
0.771812080536913	-1\\
0.785234899328859	-1\\
0.798657718120805	-1\\
0.812080536912752	-1\\
0.825503355704698	-1\\
0.838926174496644	-1\\
0.852348993288591	-1\\
0.865771812080537	-1\\
0.879194630872483	-1\\
0.892617449664429	-1\\
0.906040268456376	-1\\
0.919463087248322	-1\\
0.932885906040269	-1\\
0.946308724832215	-1\\
0.959731543624161	-1\\
0.973154362416107	-1\\
0.986577181208054	-1\\
1	-1\\
};
\addlegendentry{true initial value}

\end{axis}
\end{tikzpicture}%

\end{subfigure}
\begin{subfigure}{0.5\textwidth}
\include{test3_1}
\end{subfigure}

\caption{Initial data and approximation with $N=160$ base functions. The Gibbs phenomena is observed at the discontinuity (left). For a small set of base functions this phenomena is also visible at terminal time $T=0.6$ (right). For $80$ base functions we observe agreement with a classical finite--volume solution. 
}
\label{fig:burgermor} 
\end{figure}

In order to investigate the oscillatory behavior we show the normalized singular values of  reduced solutions for smooth and non--smooth initial data in Figure \ref{fig:sing}. The decay of the singular values  in the beginning is similar but then, as expected, the smooth initial condition shows a significant decrease in the singular values compared with the non--smooth case. This also validates the observation that a low number of base functions might not necessary be sufficient to resolve the strong discontinuities. However, as seen in Figure \ref{fig:burgermor} a model order reduction is still possible and we will investigate the found reduced basis for different initial conditions in the forthcoming section. The singular value decay of the matrix created by both solution is also shown in Figure \ref{fig:sing}. The dominant modes are also the basis of the reduced model used in the next section.

 \begin{figure}
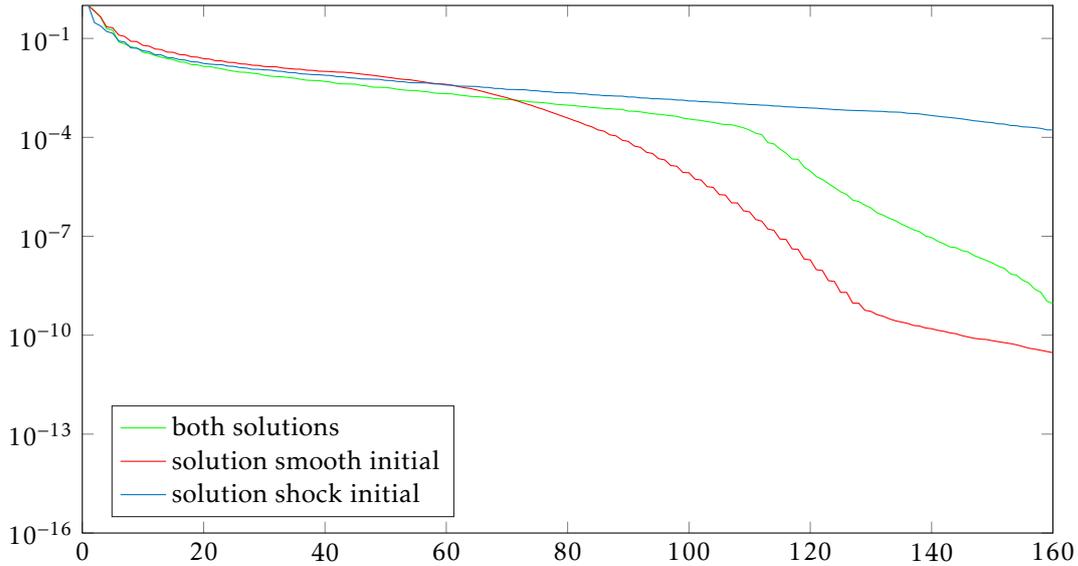

\include{sing_comp}
\caption{Singular values for the $N=160$ solution. Compared with singular values of the solution to parabolic or elliptic problems the decay is not as strong. We show the singular values for two examples: the solution to a smooth profile developing a shock (red) and for a discontinuous solution. Further, the decay for the combination both initial conditions is shown (green).  }\label{fig:sing}

\end{figure}

\subsection{Model Order Reduction for Different Initial Conditions}

We use a reduced model obtained from a combination of  the above initial conditions to predict model output for  {\em different} initial conditions. We consider the solutions to  the two different initial conditions  given in the previous section.  A reduced model from the dominant basis functions of the first two problems is obtained. The solution to this reduced system for initial data given by equation \eqref{new1} is compared with classical finite--volume integration.  The initial condition is chosen  as a linear combination of the two previous initial conditions. 

\begin{align}\label{new1} u_0=\sin(\pi x)+a\left( \chi_{0, 1/2 }(x) - 1\right) \end{align}

with $a=0.2$ Note that the solution $u(T,x)$ is {\bf not} a linear combination of the two previous solutions due to the nonlinear nature of the problem.  Hence, we  generate computational efficiency by reducing the size of the Ansatz space needed to solve for a given initial datum. Results are shown in Figure \ref{fig:test4}. The initial condition is  sinoidal with an  additional discontinuities. The reduced system solution at $T=0.3$ as well as the finite volume comparison  showing good  qualitative and quantitative agreement.  In this example we set the number of base functions as $N=160$, the dimension of the reduced space $r=80$. The further parameters are $\epsilon=10^{-3}$ as above and $\Delta t=\frac{\epsilon}{10}$. The results confirm that the chosen approach allows to efficiently apply model order reduction to hyperbolic problems.

 \begin{figure}[hbt]
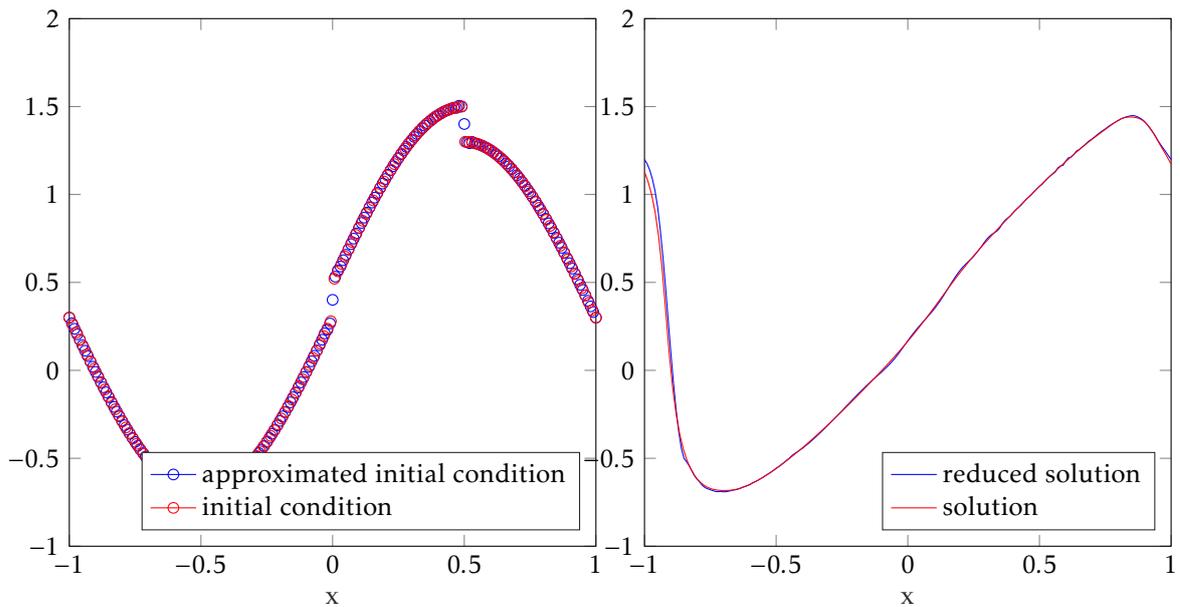

\begin{subfigure}{0.5\textwidth}
\include{test4_ini}
\end{subfigure}
\begin{subfigure}{0.5\textwidth}
\include{test4}
\end{subfigure}

\caption{Left: Initial condition on the full space and reduced space with $r=80$ out of $N=160$ base functions. Right: Solution at time $T=0.3$ on full and reduced space.
}
\label{fig:test4}
\end{figure}
\section{Summary}
We proposed a relaxation formulation of hyperbolic conservation laws that allows to use shifted base functions for a formulation that is amendable for model-order reduction. The resulting discretized scheme is reduced using snapshots in time and shows qualitative good approximation properties even in the case of shock waves. The approach has been tested on linear hyperbolic problems with nonlinear source terms but known exact solution as well as nonlinear hyperbolic problems with strong shocks. A numerical investigation of the approximation quality, the singular value decay as well as comparisons with classical finite-volume schemes have been conducted. 

\bibliographystyle{siam}
\bibliography{file,mor,papers,completeBibTex}

\subsection*{Acknowledgments}

The authors thank the Deutsche Forschungsgemeinschaft (DFG, German Research Foundation) for the financial support through 20021702/GRK2326, 333849990/IRTG-2379, HE5386/18-1,19-1,22-1 and under Germany's Excellence Strategy EXC-2023 Internet of Production 390621612. 
DFG 18,19-1. Supported also by the German Federal Ministry for Economic Affairs and Energy, in the joint project:
“MathEnergy – Mathematical Key Technologies for Evolving Energy Grids”, sub-project: Model
Order Reduction (Grant number: 0324019B).

\end{document}